\documentclass[11pt]{article}
\usepackage{latexsym}
\usepackage{amsfonts}

\def\<{\left\langle}
\def\>{\right\rangle}
\def\({\left(}
\def\){\right)}

\def\bea{\begin{eqnarray*}}
\def\eea{\end{eqnarray*}}

\newtheorem{main}{Theorem}

\newtheorem{thm}{Theorem}
\newtheorem{lem}{Lemma}
\newtheorem{prop}{Proposition}
\newtheorem{defn}{Definition}

\newtheorem{cor}{Corollary}
\newtheorem{rem}{Remark}

\newenvironment{proof}{\medskip \noindent
{\bf Proof.}}{\hfill \rule{.5em}{1em}
\\}

\begin{document}
\sloppy
\title{Harnack Estimates for Nonlinear Backward Heat Equations in Geometric Flows}

\author{Hongxin Guo and Masashi Ishida}

\date{}

\maketitle

\begin{abstract}
Let $M$ be a closed Riemannian manifold with a family of Riemannian metrics $g_{ij}(t)$ evolving by a geometric flow $\partial_{t}g_{ij} = -2{S}_{ij}$, where $S_{ij}(t)$ is a family of smooth symmetric two-tensors. We derive several differential Harnack estimates for positive solutions to the nonlinear backward heat-type equation
\begin{eqnarray*}
\frac{\partial f}{\partial t} = -{\Delta}f + \gamma f\log f +aSf
\end{eqnarray*}
where $a$ and $\gamma$ are constants and $S=g^{ij}S_{ij}$ is the trace of $S_{ij}$.
Our abstract formulation provides a unified framework for some known results proved by various authors,
and moreover lead to new Harnack inequalities for a variety of geometric flows
.
\end{abstract}

%%%%%%%%%%%%%%%%%%%%%%%%%%%%%%%%%%%%%%%%%%%%%%%%%%%%%%%%%%%%%%%%
\section{Introduction}\label{intro}
%%%%%%%%%%%%%%%%%%%%%%%%%%%%%%%%%%%%%%%%%%%%%%%%%%%%%%%%%%%%%%%%
The study of differential Harnack estimates for parabolic equations originated with the work of Li and Yau \cite{Li-Yau}. They proved a gradient estimate for the heat equation by using the maximal principle. By integrating the gradient estimate along a space-time path, a classical Harnack inequality was derived. Therefore, Li-Yau type gradient estimate is often called differential Harnack estimate.  Similar techniques were used by Hamilton to prove Harnack estimates for the Ricci flow \cite{h-surface, ha-1} and the mean curvature flow \cite{ha-2}.

Using similar techniques, many authors have proved a variety of Li-Yau-Hamilton's Harnack estimates for various equations
in different geometric flows, and we refer to the survey paper by Ni \cite{Ni-lei}.
In the Ricci flow, Perelman proved a Harnack inequality for the fundamental solution of the conjugate heat equation \cite{p-1}
and later on different Harnack inequalities have been proved,
to name but a few \cite{C, C-H, C-Z, K-Z, Liu-1, p-1, Wu, Z-1}. Harnack inequalities for more general flows have
been considered in \cite{S-F, G-F, G-H, Ishida-1, Sun}. \par

The main purpose of the current article is, in the framework of a general geometric flow,
to derive Li-Yau-Hamilton type differential Harnack estimates
for positive solutions to a nonlinear backward heat-type equation
generalizing Perelman's conjugate heat equation.

Let $M$ be a closed Riemannian $n$-manifold with a one parameter family of Riemannian metrics $g(t)$ evolving by the geometric flow
\begin{eqnarray}\label{GF}
\frac{\partial}{\partial t}g_{ij} = - 2S_{ij} \ , \  t \in [0, T)
\end{eqnarray}
where $S_{ij}(t)$ is any smooth symmetric two-tensor on $(M, g(t))$.
$f$ is a positive solution to
\begin{eqnarray}\label{backward-non}
\frac{\partial f}{\partial t} = -{\Delta}f + \gamma f\log f + aSf.
\end{eqnarray}
where symbol $\Delta$ stands for the Laplacian of the evolving metric $g(t)$, $\gamma$ and $a$ are constants and $S=g^{ij}S_{ij}$ is the trace of $S_{ij}$. In the Ricci flow case, when $\gamma=0$ and $a=1$, (\ref{backward-non}) is the conjugate heat equation introduced by Perelman. The consideration of this nonlinear equation is motivated by gradient Ricci solitons. See \cite{C-Z, Wu} for more details.
In a forthcoming paper \cite{G-I}, we will consider the forward nonlinear heat-type equation.

To state the main results, we introduce evolving tensor quantities associated to the tensor ${S}_{ij}$.
\begin{defn}\label{h-d-de}
Let $g(t)$ evolve by (\ref{GF}) and $X=X^{i}\frac{\partial}{\partial x^{i}}$ be a vector field on $M$.
For a constant $a \in {\mathbb R}$, we define
\begin{eqnarray}\label{E-definition}
{\mathcal E}_{a}(S_{ij}, {X})= \Big(a\frac{\partial S}{\partial \tau} + a\Delta S + 2 |S_{ij}|^{2} \Big) - 2 \Big(2{\nabla}^{i}S_{i\ell} - {\nabla}_{\ell}S \Big){X}^{\ell} - 2(R^{ij} -S^{ij}){X}_{i}{X}_{j},
\end{eqnarray}
where $\tau = T -t$, $R^{ij} = {g}^{ik}g^{j \ell}R_{k \ell}$, $S^{ij} = {g}^{ik}g^{j \ell}S_{k \ell}$, $S = g^{ij}S_{ij}$, ${\nabla}^{i} = g^{ij}{\nabla}_{j}$ and $X_{k} = {g}_{ik}X^{i}$.
\end{defn}

\begin{rem}\label{the remark for E}
The quantity of ${\mathcal E}_{a}(S_{ij}, {X})$ is a generalization of $\mathcal D (S_{ij}, X)$ defined by Reto M\"uller \cite{Mu-1}.
Indeed, $\mathcal D (S_{ij}, X)=-{\mathcal E}_{1}(S_{ij}, {X})$. ${\mathcal E}_{a}(S_{ij}, {X})$ has also been implicitly discussed
in recent papers of one of the authors and others \cite{G-H, G-P-T}.  ${\mathcal E}_{a}(S_{ij}, {X})$ appears as the error term
in our formulas under the flow (\ref{GF}), and when the flow is Hamilton's Ricci flow ${\mathcal E}_{1}(R_{ij}, {X})=0$.
\end{rem}

In the following theorems A-E, we assume $(M, g(t))$, $t \in [0, T)$, is a solution to the geometric flow (\ref{GF}) on a closed oriented smooth $n$-manifold $M$. Denote $\tau:=T-t$.

\begin{main}\label{main-A}
\begin{enumerate}
\item Suppose that
\begin{eqnarray}\label{key-cond-2}
{\mathcal E}_{2}(S_{ij},X) + \frac{2S^{2}}{n} \leq 0
\end{eqnarray}
holds for all vector fields $X$ and all time $t \in [0, T)$ for which the flow exists.
Let $f$ be a positive solution to the equation (\ref{backward-non}) with $\gamma=1$ and $a=2$,
\begin{eqnarray*}
\frac{\partial f}{\partial t} = -{\Delta}f + f\log f + 2Sf.
\end{eqnarray*}
Then, for all time $t \in [0, T)$
it holds
\begin{eqnarray*}
 2{\Delta}\log f +|\nabla \log f|^{2}-2{S} +2 \frac{n}{\tau} + \frac{n}{2}\geq 0.
\end{eqnarray*}
\item Suppose that
\begin{eqnarray}\label{key-cond-3}
{\mathcal E}_{1}(S_{ij},X) \leq 0, \ S \geq 0
\end{eqnarray}
hold for all vector fields $X$ and all time $t \in [0, T)$ for which the flow exists.
Let $f$ be a positive solution to the equation (\ref{backward-non}) with $\gamma=1$ and $a=1$, namely
\begin{eqnarray*}
\frac{\partial f}{\partial t} = -{\Delta}f + f\log f + Sf.
\end{eqnarray*}
Then, for all time $t \in [0, T)$ it holds
\begin{eqnarray*}
 2{\Delta}\log f + |\nabla \log f|^{2}-{S} +2 \frac{n}{\tau} + \frac{n}{4}\ge 0.
\end{eqnarray*}
\end{enumerate}
\end{main}
In the second case of Theorem \ref{main-A}, dropping the assumption $S\ge 0$ we are able to prove:
\begin{main}\label{main-Aa}
Let $f$ be a positive solution to
\begin{eqnarray*}
\frac{\partial f}{\partial t} = -{\Delta}f + f\log f + Sf.
\end{eqnarray*}
Suppose that
\begin{eqnarray}\label{s-t-bound}
{\mathcal E}_{1}(S_{ij},X) \leq 0, \ S \geq -\frac{n}{2t}
\end{eqnarray}
hold for all vector fields $X$ and all time $t \in [\frac{T}{2}, T)$ for which the flow exists.
Then for all time $t \in [\frac{T}{2}, T)$, it holds
\begin{eqnarray*}
2{\Delta}\log f + |\nabla \log f|^{2}-{S} +3 \frac{n}{\tau} + \frac{n}{4}\ge 0.
\end{eqnarray*}
\end{main}
We are able to derive classical Harnack inequalities by integrating the above differential Harnack inequalities along space-time paths. For instance, the first case of Theorem \ref{main-A} implies following. The proof is now standard (for example, see \cite{C, Wu}).
\begin{cor}
Suppose that (\ref{key-cond-2}) holds for all vector fields $X$ and all time $t \in [0, T)$ for which the flow exists. Let $f$ be a positive solution to the heat equation
\begin{eqnarray*}
\frac{\partial f}{\partial t} = -{\Delta}f + f\log f + 2Sf.
\end{eqnarray*}
Assume that $(x_{1}, t_{1})$ and $(x_{2}, t_{2})$ are two points in $M \times (0, T)$, where $0 < t_{1} < {t}_{2} < T$. Then the following holds:
\begin{eqnarray*}
 e^{t_{2}} \log f(x_{2}, t_{2})-e^{t_{1}} \log f(x_{1}, t_{1}) \leq \frac{1}{2} {\int}^{t_{2}}_{t_{1}}
e^{\tau} \Big(|\dot{\ell}|^{2} +2S + \frac{n}{2} +2\frac{n}{\tau} \Big)dt,
\end{eqnarray*}
where $\ell$ is any space-time path joining $(x_{1}, t_{1})$ and $(x_{2}, t_{2})$.
\end{cor}
We notice that the second case of Theorem \ref{main-A} and Theorem \ref{main-Aa} also imply
similar classical Harnack inequalities. We leave details to the interested readers.\par

On the other hand, for the equation (\ref{backward-non}) in the case where $a=0$, we shall prove the following result:
\begin{main}\label{main-B}
Suppose $\gamma > 0$ and $0<f<1$ is a positive solution to
\begin{eqnarray*}
\frac{\partial f}{\partial t} = -{\Delta}_{}f +\gamma f\log f.
\end{eqnarray*}
If
\begin{eqnarray}\label{R-S-A}
R_{ij}(t) + S_{ij}(t) \geq -\frac{1}{2}\gamma g_{ij}(t),
\end{eqnarray}
then for all time $t \in [0, T)$ it holds:
\begin{eqnarray*}
|\nabla \log f|^{2} + \frac{\log f}{\tau} \leq 0.
\end{eqnarray*}
\end{main}
For the equation (\ref{backward-non}) in the case where $a=0$ and $\gamma =0$, we prove
\begin{main}\label{main-C}
Suppose that 
\begin{eqnarray}\label{assumption in main-C}
R_{ij}(t) + S_{ij}(t) \geq 0.
\end{eqnarray}
Let $0<f<1$ be a positive solution to
the time-dependant heat equation
\begin{eqnarray*}
\frac{\partial f}{\partial t} = -{\Delta}_{}f.
\end{eqnarray*}
Then for all time $t \in [0, T)$, the following holds:
\begin{eqnarray*}
|\nabla \log f|^{2} + \frac{\log f}{\tau} \leq 0.
\end{eqnarray*}
\end{main}
Moreover, we also prove
\begin{main}\label{main-E}
Let $f$ be a positive solution to the conjugate heat equation
\begin{eqnarray*}
\frac{\partial f}{\partial t} = -{\Delta}f + Sf.
\end{eqnarray*}
Suppose that
\begin{eqnarray}\label{monotone-con}
{\mathcal E}_{1}(S_{ij},X) \leq 0
\end{eqnarray}
holds for all vector fields $X$ and all time $t \in [0, T)$ for which the flow exists.
Then
$$
\displaystyle \min_{M} \Big( 2 \Delta \log f + |\nabla \log f|^{2} - {S} \Big)
$$
increases along the geometric flow (\ref{GF}).
\end{main}

The rest of this article is organized as follows. In Section \ref{ex} we apply the abstract results, Theorems A-E, to a variety of geometric flows. We shall justify that the technical assumptions (\ref{key-cond-2}, \ref{key-cond-3}, \ref{s-t-bound},  \ref{R-S-A}, \ref{assumption in main-C}, \ref{monotone-con}) are either automatically satisfied, or guaranteed by the geometric assumption at time $t=0$. A variety of new Harnack inequalities are obtained. In Section \ref{GEE}, we shall derive general evolution equations. Sections \ref{POA} to \ref{section-E}
are devoted to proving Theorems A-E.

\vspace{0.2cm}
\noindent
{\bf Acknowledgements}
The first author was supported by NSFC (Grants No. 11001203 and 11171143) and Zhejiang Provincial Natural Science Foundation of China (Project No. LY13A010009). The second author was partially supported by the Grant-in-Aid for Scientific Research (C), Japan Society for the Promotion of Science, No. 25400074.

%%%%%%%%%%%%%%%%%%%%%%%%%%%%%%%%%%%%%%%%%%%%
\section{Examples}\label{ex}
%%%%%%%%%%%%%%%%%%%%%%%%%%%%%%%%%%%%%%%%%%%%
\noindent{\bf{(0) Static Riemannian manifold.}} In this case $S_{ij}=0$ and thus for any $a$
$$
{\mathcal E}_{a}(0, {X})=-2R^{ij}X_iX_j
$$
As long as $(M, g)$ has nonnegative Ricci curvature it holds ${\mathcal E}_{a}(0, {X})\le 0$
and moreover the technical assumptions (\ref{key-cond-2}, \ref{key-cond-3}, \ref{s-t-bound}, \ref{R-S-A}, \ref{monotone-con}) all hold.
Applying our main theorems to this example, and reversing the time direction we have:
\begin{cor}\label{nonlinear-static}
Suppose $(M, g)$ is a compact static Riemannian manifold with nonnegative Ricci curvature. Let
$f$ be a positive solution to
\begin{eqnarray}\label{nonlinear heat equation-static}
\frac{\partial f}{\partial t} = {\Delta}f - f\log f.
\end{eqnarray}
Then for all $t>0$
$$
2\Delta\log f +|\nabla\log f|^2+\frac{2n}{t}+\frac {n}{4}\ge 0.
$$
\end{cor}

We note that gradient estimates of the corresponding elliptic version of (\ref{nonlinear heat equation-static}) have been studied in \cite{Ma}.\par

%When a Riemannian manifold is compact, the Ricci curvature has always a lower bound. Thus 
By Theorem \ref{main-C} and \ref{main-E} and noting
the time direction, we have
\begin{cor}\label{linear-static}
Suppose $(M, g)$ is a compact static Riemannian manifold with nonnegative Ricci curvature. Let
$0<f<1$ be a positive solution to
\begin{eqnarray*}
\frac{\partial f}{\partial t} = {\Delta}f.
\end{eqnarray*}
Then, for all $t$ it holds
$$
|\nabla\log f|^2+\frac{\log f}{t}\le 0.
$$
%If in addition the Ricci curvature is nonnegative then
And 
$$
\min_M\{2\Delta\log f+|\nabla\log f|^2\}
$$
decreases along time.
\end{cor}

\noindent{\bf{(1) Hamilton's Ricci flow.}} Let $g(t)$ be a solution to the Ricci flow:
\begin{eqnarray*}
\frac{\partial}{\partial t}g_{ij} = - 2R_{ij}.
\end{eqnarray*}
Namely, we have $S_{ij} = {R}_{ij}$ and $S = R$ the scalar curvature. Notice that it is known that the scalar curvature $R$ evolves by $\frac{\partial R}{\partial t} - \Delta R - 2|R_{ij}|^{2} = 0$. Therefore
$$
\frac{\partial R}{\partial \tau} = - \Delta R - 2|R_{ij}|^{2}.
$$
Moreover, we have the twice contracted second Bianchi identity $2{\nabla}^{i}R_{i\ell} - {\nabla}_{\ell}R=0$. Hence, we have
\begin{eqnarray*}
{\mathcal E}_{a}(R_{ij}, {\nabla}v) =   \left(a\frac{\partial R}{\partial \tau} + a\Delta R + 2|R_{ij}|^{2} \right)
= -2(a-1)|R_{ij}|^{2}.
\end{eqnarray*}
This implies that ${\mathcal E}_{a}(R_{ij}, {\nabla}v) \leq 0$ if $a \geq 1$. In particular, (\ref{monotone-con}) holds automatically. 
Moreover, (\ref{key-cond-2}) is equivalent to $|R_{ij}|^{2} \geq  \frac{R^{2}}{n}$ which is automatically satisfied.
Moreover, (\ref{s-t-bound}) is equivalent to
\begin{eqnarray}\label{R geq t}
R \geq -\frac{n}{2t}.
\end{eqnarray}
Notice that the evolution of scalar curvature $R$ under the Ricci flow satisfies
$$
\frac{\partial R}{\partial t} = \Delta R + 2|R_{ij}|^{2} \geq \Delta R + \frac{2}{n}R^2.
$$
By the maximal principle, we get (\ref{R geq t}). 

Summarizing, in Hamilton's Ricci flow we have 
(\ref{key-cond-2}, \ref{s-t-bound}, \ref{monotone-con}) hold
automatically. (\ref{key-cond-3})
holds when $R(0)\ge 0$, and (\ref{R-S-A}, \ref{assumption in main-C})
can be guaranteed  if the curvature operator at time $t=0$ is nonnegative.
We remark that in the Ricci flow the corresponding results have been proved in \cite{C, C-Z, Wu}.

\vspace{0.3cm}
\noindent{\bf{(2) List's extended Ricci flow.}} In \cite{List-1}, List introduced a geometric flow closely related to the Ricci flow:
\begin{eqnarray*}
\frac{\partial}{\partial t}g_{ij} &=& - 2R_{ij} + 4{\nabla}_i \psi {\nabla}_j \psi, \\
\frac{\partial \psi}{\partial t} &=& {\Delta} \psi,
\end{eqnarray*}
where $\psi : M \rightarrow {\mathbb R}$ is a smooth function. In the extended Ricci flow $S_{ij} = {R}_{ij}-2{\nabla}_i \psi {\nabla}_j \psi$ and $S =R -2|\nabla \psi|^{2}$. List \cite{List-1} pointed out that $S$ satisfies the following evolution equation:
\begin{eqnarray*}
\frac{\partial S}{\partial t} = \Delta S + 2|S_{ij}|^{2} + 4|\Delta \psi|^{2}.
\end{eqnarray*}
Therefore,
\begin{eqnarray*}
a\frac{\partial S}{\partial \tau} + a\Delta S + 2 |S_{ij}|^{2}=  - 2(a-1)|S_{ij}|^{2} - 4a|\Delta \psi|^{2}.
\end{eqnarray*}
On the other hand, we have (see \cite{List-1} and Section 2 in \cite{Mu-1}):
$2{\nabla}^{i}S_{i\ell} - {\nabla}_{\ell}S = -4{\Delta}\psi \nabla_{\ell} \psi.$
Therefore we obtain
\begin{eqnarray*}
{\mathcal E}_{a}(S_{ij}, {X})=- 2(a-1)|S_{ij}|^{2} - 4(a-1)|\Delta \psi|^{2} -4 |\Delta \psi - {\nabla}_{X} \psi|^{2}.
\end{eqnarray*}
This tells that ${\mathcal E}_{a}(S_{ij}, {X}) \leq 0$ if $a \geq 1$. \\
On the other hand, if $a=2$, then ${\mathcal E}_{2}(S_{ij}, {X})= - 2|S_{ij}|^{2} - 4|\Delta \psi|^{2} -4 |\Delta \psi - {\nabla}_{X} \psi|^{2}$
and we get the following:
\begin{eqnarray*}
{\mathcal E}_{2}(S_{ij}, {X}) + \frac{2}{n}S^{2} & = & - 2|S_{ij}|^{2} + \frac{2}{n}S^{2} - 4|\Delta \psi|^{2} -4 |\Delta \psi - {\nabla}_{X} \psi|^{2} \\
& \leq & - \frac{2}{n}S^{2} + \frac{2}{n}S^{2} - 4|\Delta \psi|^{2} -4 |\Delta \psi - {\nabla}_{X} \psi|^{2} \\
&=& - 4|\Delta \psi|^{2} -4 |\Delta \psi - {\nabla}_{X} \psi|^{2} \\
&\leq& 0.
\end{eqnarray*}
Namely, (\ref{key-cond-2}) holds. 
Similar as in the Ricci flow, we get
\begin{eqnarray*}
\frac{\partial S}{\partial t} = \Delta S + 2|S_{ij}|^{2} + 4|\Delta \psi|^{2} &\geq& \Delta S + \frac{2}{n}|S|^{2} + 4|\Delta \psi|^{2} \\
&\geq& \Delta S + \frac{2}{n}|S|^{2}
\end{eqnarray*}
and by applying the maximal principle, we get
\begin{eqnarray*}
S \geq -\frac{n}{2t}.
\end{eqnarray*}
Summarizing, in List's extended Ricci flow we have
(\ref{key-cond-2}, \ref{s-t-bound}, \ref{monotone-con}) hold
automatically. Notice that the positivity of $S = R - 2|\nabla \psi|^{2}$ is preserved by the flow, we get (\ref{key-cond-3})
holds when $S(0)\ge 0$.

\vspace{0.3cm}
\noindent{\bf{(3) M{\"{u}}ller's Ricci flow coupled with harmonic map flow.}} Let $(Y, h)$ be a fixed Riemannian manifold. Let $(g(t), \phi(t))$ be the couple consisting of a family of metric $g(t)$ on $M$ and a family of maps $\phi(t)$ from $M$ to $Y$. We call  $(g(t), \phi(t))$ a solution of M{\"{u}}ller's flow \cite{Mu-2} (also known as Ricci flow coupled with harmonic map heat flow) with coupling function $\alpha(t) \geq 0$ if
\begin{eqnarray*}
\frac{\partial}{\partial t}g_{ij} &=& - 2R_{ij} + 2\alpha(t){\nabla}_i \phi {\nabla}_j \phi, \\
\frac{\partial \phi}{\partial t} &=& {\tau} _g \phi,
\end{eqnarray*}
where $ {\tau} _g \phi$ is the tension field of the map $\phi$ with respect to the metric $g(t)$. List's flow is a special case of this flow. 
In this example $S_{ij} = {R}_{ij}-\alpha(t){\nabla}_i \phi {\nabla}_j \phi$, and $S =R -\alpha(t)|\nabla \phi|^{2}$. 
M{\"{u}}ller \cite{Mu-2} proved that $S$ satisfies
\begin{eqnarray*}
\frac{\partial S}{\partial t} = \Delta S + 2|S_{ij}|^{2} + 2\alpha(t)|\tau_{g} \phi|^{2} -(\frac{\partial \alpha(t)}{\partial t})|\nabla \phi|^{2}.
\end{eqnarray*}
Therefore, we get
\begin{eqnarray*}
a\frac{\partial S}{\partial \tau} + a\Delta S + 2 |S_{ij}|^{2}=  - 2(a-1)|S_{ij}|^{2}-2a\alpha(t)|\tau_{g} \phi|^{2} +a (\frac{\partial \alpha(t)}{\partial t})|\nabla \phi|^{2}.
\end{eqnarray*}
On the other hand, we have (see \cite{Mu-2} and Section 2 in \cite{Mu-1}):
\begin{eqnarray*}
2\Big( 2{\nabla}^{i}S_{i\ell} - {\nabla}_{\ell}S\Big) {X}^{\ell} = -4\alpha (t) {\tau_{g}}\psi \nabla_{\ell} \phi{X}^{\ell}
\end{eqnarray*}
and moreover
\begin{eqnarray*}
{\mathcal E}_{a}(S_{ij}, {X}) &=&- 2(a-1)|S_{ij}|^{2}-2a\alpha(t)|\tau_{g} \phi|^{2} +a (\frac{\partial \alpha(t)}{\partial t})|\nabla \phi|^{2} + 4\alpha (t) {\tau_{g}}\psi \nabla_{\ell} \phi{X}^{\ell} \\
&-& 2\alpha (t) {\nabla}_i \phi {\nabla}_j \phi{X}^{i}{X}^{j} \\
&=& - 2(a-1)|S_{ij}|^{2}-2\alpha(t)(a-1)|\tau_{g} \phi|^{2} + a (\frac{\partial \alpha(t)}{\partial t})|\nabla \phi|^{2} \\
&-&2\alpha(t) \Big(|\tau_{g} \phi|^{2}-2{\tau_{g}}\phi \nabla_{\ell} \phi{X}^{\ell} + {\nabla}_i \phi {\nabla}_j \psi{X}^{i}{X}^{j}\Big) \\
&=& - 2(a-1)|S_{ij}|^{2} -2\alpha(t)(a-1)|\tau_{g} \phi|^{2} -2\alpha (t) |\tau_{g} \phi - {\nabla}_{X} \phi|^{2} + a (\frac{\partial \alpha(t)}{\partial t})|\nabla \phi|^{2}.
\end{eqnarray*}
Suppose that $a \geq 1$ and $\frac{\partial \alpha(t)}{\partial t} \leq 0$, we get
\begin{eqnarray*}
{\mathcal E}_{a}(S_{ij}, {X}) \leq 0.
\end{eqnarray*}
In particular, (\ref{monotone-con}) holds automatically. Moreover, (\ref{key-cond-3}) is equivalent to $S \geq 0$.

On the other hand, suppose that $a=2$ and $\frac{\partial \alpha(t)}{\partial t} \leq 0$. Then the above computation tells us that
$$
{\mathcal E}_{2}(S_{ij}, {X}) = - 2|S_{ij}|^{2} -2\alpha(t)|\tau_{g} \phi|^{2} -2\alpha (t) |\tau_{g} \phi - {\nabla}_{X} \phi|^{2} + 2 (\frac{\partial \alpha(t)}{\partial t})|\nabla \phi|^{2}.
$$
Hence
\begin{eqnarray*}
{\mathcal E}_{2}(S_{ij}, {X}) + \frac{2}{n}S^{2} & = & - 2|S_{ij}|^{2} + \frac{2}{n}S^{2} -2\alpha(t)|\tau_{g} \phi|^{2} -2\alpha (t) |\tau_{g} \phi - {\nabla}_{X} \phi|^{2} + 2 (\frac{\partial \alpha(t)}{\partial t})|\nabla \phi|^{2} \\
& \leq & - \frac{2}{n}S^{2} + \frac{2}{n}S^{2} -2\alpha(t)|\tau_{g} \phi|^{2} -2\alpha (t) |\tau_{g} \phi - {\nabla}_{X} \phi|^{2} + 2 (\frac{\partial \alpha(t)}{\partial t})|\nabla \phi|^{2} \\
&=& -2\alpha(t) \Big(|\tau_{g} \phi|^{2} + |\tau_{g} \phi - {\nabla}_{X} \phi|^{2} \Big) + 2 (\frac{\partial \alpha(t)}{\partial t})|\nabla \phi|^{2} \\
&\leq& 0,
\end{eqnarray*}
where notice that $\alpha(t) \geq 0$ and $\frac{\partial \alpha(t)}{\partial t} \leq 0$. Namely, (\ref{key-cond-2}) holds. 

On the other hand, we get the following under $\alpha (t) \geq 0$ and $\frac{\partial \alpha(t)}{\partial t} \leq 0$:
\begin{eqnarray*}
\frac{\partial S}{\partial t} &=& \Delta S + 2|S_{ij}|^{2} + 2\alpha(t)|\tau_{g} \phi|^{2} -(\frac{\partial \alpha(t)}{\partial t})|\nabla \phi|^{2} \\
&\geq & \Delta S + \frac{2}{n}|S_{ij}|^{2} + 2\alpha(t)|\tau_{g} \phi|^{2} -(\frac{\partial \alpha(t)}{\partial t})|\nabla \phi|^{2} \\
&\geq & \Delta S + \frac{2}{n}|S_{ij}|^{2}.
\end{eqnarray*}
By applying the maximal principle, we get
\begin{eqnarray*}
S \geq -\frac{n}{2t}.
\end{eqnarray*}
Therefore (\ref{s-t-bound}) holds, where notice that we have ${\mathcal E}_{1}(S_{ij}, {X}) \leq 0$. \\
Summarizing, in M\"{u}ller's flow we have
(\ref{key-cond-2}, \ref{s-t-bound}, \ref{monotone-con}) hold
automatically. Notice that the positivity of $S$ is again preserved by the flow, we get (\ref{key-cond-3})
holds when $S(0)\ge 0$.

%%%%%%%%%%%%%%%%%%%%%%%%%%%%%%%%%%%%%%%%%%%%
\section{General evolution equations}\label{GEE}
%%%%%%%%%%%%%%%%%%%%%%%%%%%%%%%%%%%%%%%%%%%%

In this section, we shall prove general evolution equations of general Harnack quantities under the geometric flow, which are useful to prove the main results in the backward case. See Theorems \ref{T-1} and \ref{T-2} stated below. In the Ricci flow case, such general evolution equations are firstly proved by Cao \cite{C}. Theorems \ref{T-1} and \ref{T-2} can be seen as generalizations of Lemma 2.1 and Lemma 3.4 in \cite{C} respectively.

\subsection{Case of $u = - \log f$}\label{sub-u}

Let $M$ be a closed Riemannian manifold with a Riemannian metric $g_{ij}(t)$ evolving by a geometric flow $\partial_{t}g_{ij} = -2{S}_{ij}$. Let $f$ be a positive solution of the following nonlinear backward heat equation with potential term $-cS$:
\begin{eqnarray*}
\frac{\partial f}{\partial t} = -{\Delta}_{}f + \gamma f\log f- cSf,
\end{eqnarray*}
where $c$ and $\gamma$ are constants. In what follows, let $u = - \log f$. By a direct computation, we see that $u$ satsifies
\begin{eqnarray}\label{u-differ-1}
\frac{\partial u}{\partial t} = -\Delta u + |\nabla u|^{2} + cS + \gamma u.
\end{eqnarray}
Let $\tau = T-t$. Then $f$ satisfies
\begin{eqnarray*}
\frac{\partial f}{\partial \tau} = {\Delta}_{}f + cSf - \gamma f\log f
\end{eqnarray*}
and $u$ satisfies
\begin{eqnarray}\label{u-differ}
\frac{\partial u}{\partial \tau} = \Delta u - |\nabla u|^{2} - cS - \gamma u.
\end{eqnarray}
On the other hand, it is known that $\Delta(|\nabla u|^{2}) = 2{\nabla}^{i}u \Delta ({\nabla}_{i} u) + 2|\nabla \nabla u|^{2}$. Since we also have $\Delta ({\nabla}_{i} u) = {\nabla}_{i}({\Delta u}) + {R}_{ij}{\nabla}^{j}u$, we obtain $\Delta(|\nabla u|^{2}) = 2{\nabla}^{i}u ({\nabla}_{i}({\Delta u}) + {R}_{ij}{\nabla}^{j}u) + 2|\nabla \nabla u|^{2}$.
Equivalently, we get
\begin{eqnarray}\label{r-formula}
2{\nabla}^{i}({\Delta}u){\nabla}_{i}u &=& {\Delta}(|\nabla u|^{2})-2{R}_{ij}{\nabla}^{i}u{\nabla}^{j}u -2|\nabla \nabla u|^{2}.
\end{eqnarray}
Then we have
\begin{lem}\label{lem-1-com}
Under the above situation, the following holds:
\begin{eqnarray*}\label{lem-1-1}
\frac{\partial}{\partial \tau}({\Delta}u) &=& \Delta (\Delta u) - \Delta(|\nabla u|^{2}) -c\Delta S - 2{S}^{ij}{\nabla}_{i}{\nabla}_{j}u -  \Big(2{\nabla}^{i}S_{i\ell} - {\nabla}_{\ell}S \Big){\nabla}^{\ell}u \\
&-& \gamma \Delta u, \\
\frac{\partial}{\partial \tau}(|\nabla u|^{2}) &=& \Delta(|\nabla u|^{2}) -2|\nabla \nabla u|^{2} -2{\nabla}^{i}(|\nabla u|^{2}){\nabla}_{i}u -2c{\nabla}^{i}S{\nabla}_{i}u \\
&-& 2(R^{ij} + S^{ij}){\nabla}_{i}u{\nabla}_{j}u - 2 \gamma|\nabla u|^{2}.
\end{eqnarray*}
\end{lem}
\begin{proof}
These formulas follow from direct computations as follows. First of all, recall that
\begin{eqnarray*}
{\Gamma}^{k}_{ij} = \frac{1}{2}g^{k \ell} \Big(\frac{\partial}{\partial {x}^{i}}g_{j\ell}  + \frac{\partial}{\partial {x}^{j}}g_{i\ell} - \frac{\partial}{\partial {x}^{\ell}}g_{ij} \Big).
\end{eqnarray*}
By working with a normal coordinate, we obtain
\begin{eqnarray*}
\frac{\partial}{\partial t}{\Gamma}^{k}_{ij} &=& \frac{1}{2}g^{k \ell} \Big({\nabla}_{i}(-2{S}_{j{\ell}} + {\nabla}_{j}(-2{S}_{i{\ell}}) - {\nabla}_{\ell}(-2{S}_{ij}) \Big) \\
&=& -{g}^{k \ell} \Big({\nabla}_{i}{S}_{j{\ell}} + {\nabla}_{j}{S}_{i{\ell}} - {\nabla}_{\ell}{S}_{ij} \Big).
\end{eqnarray*}
Therefore, the following holds:
\begin{eqnarray*}
g^{ij} \Big(\frac{\partial}{\partial t}{\Gamma}^{k}_{ij}\Big) &=& -g^{ij} {g}^{k \ell} \Big({\nabla}_{i}{S}_{j{\ell}} + {\nabla}_{j}{S}_{i{\ell}} - {\nabla}_{\ell}{S}_{ij} \Big) = - {g}^{k \ell} \Big({\nabla}^{j}{S}_{j{\ell}} + {\nabla}^{i}{S}_{i{\ell}} - {\nabla}_{\ell}{S} \Big) \\
&=& -{g}^{k \ell} \Big( 2{\nabla}^{i}S_{i\ell}- {\nabla}_{\ell}{S} \Big).
\end{eqnarray*}
By using this and (\ref{u-differ-1}), we have
\begin{eqnarray*}\label{f-1}
\frac{\partial}{\partial \tau} (\Delta u) &=& -2{S}^{ij}{\nabla}_{i}{\nabla}_{j}u - \Delta(\frac{\partial u}{\partial t}) + g^{ij}\Big(\frac{\partial}{\partial t}\Gamma^{k}_{ij} \Big)\nabla_{k}u \\
&=& -2{S}^{ij}{\nabla}_{i}{\nabla}_{j}u - \Delta(-\Delta u + |\nabla u|^{2} + cS + \gamma u ) \\
&-& g^{k \ell} (2\nabla^{i}S_{i\ell} - \nabla_{\ell}S) \nabla_{k}u \\
&=& \Delta (\Delta u) - \Delta(|\nabla u|^{2}) -c\Delta S -\gamma \Delta u - 2{S}^{ij}{\nabla}_{i}{\nabla}_{j}u \\
&-& \Big(2{\nabla}^{i}S_{i\ell} - {\nabla}_{\ell}S \Big){\nabla}^{\ell}u.
\end{eqnarray*}
On the other hand, we get
\begin{eqnarray*}
\frac{\partial}{\partial \tau}(|\nabla u|^{2}) &=& -2{S}^{ij}{\nabla}_{i}u{\nabla}_{j}u - 2 \nabla^{i}(\frac{\partial u}{\partial t}) \nabla_{i} u \\
&=& -2{S}^{ij}{\nabla}_{i}u{\nabla}_{j}u - 2 \nabla^{i}(-\Delta u + |\nabla u|^{2} + cS + \gamma u) \nabla_{i} u \\
&=& -2{S}^{ij}{\nabla}_{i}u{\nabla}_{j}u + 2 \nabla^{i}(\Delta u)\nabla_{i} u - 2 \nabla^{i}(|\nabla u|^{2})\nabla_{i} u \\
&-& 2c \nabla^{i}S\nabla_{i} u - 2 \gamma|\nabla u|^{2},
\end{eqnarray*}
where we used (\ref{u-differ-1}). Since we have (\ref{r-formula}), we obtain the desired formula.
\end{proof}

By using Lemma \ref{lem-1-com}, we prove the following which is used to prove Theorems \ref{main-B} and \ref{main-C}:
\begin{prop}\label{lem-key-1}
Let $g(t)$ be a solution to the geometric flow (\ref{GF}) and $u$ satisfies (\ref{u-differ}). Let
\begin{eqnarray*}\label{Def-H-1a}
H_{S} = \alpha \Delta u - \beta |\nabla u|^{2} + a{S} + b\frac{u}{\tau} + d\frac{n}{\tau},
\end{eqnarray*}
where $\alpha, \beta, a, b$ and $d$ are constants. Then $H_{S}$ satisfies the following evolution equatios:
\begin{eqnarray*}
\frac{\partial H_{S}}{\partial \tau}
&=& {\Delta}H_{S} - 2{\nabla}^{i}H_{S}{\nabla}_{i}u + 2(a+\beta c){\nabla}^{i}S{\nabla}_{i}u-2(\alpha-\beta)|\nabla \nabla u|^{2} \\
&-&2{\alpha}S^{ij}{\nabla}_{i}{\nabla}_{j}u  + \frac{b}{\tau}|\nabla u|^{2} - \frac{b}{\tau}cS - \frac{b}{\tau^{2}}u - d\frac{n}{\tau^{2}} \\
&+& a\frac{\partial S}{\partial \tau} - \Big(a + \alpha c \Big) \Delta S - \alpha \Big(2{\nabla}^{i}S_{i\ell} - {\nabla}_{\ell}S \Big){\nabla}^{\ell}u \\
&-& 2\Big(\alpha R^{ij} - \beta(R^{ij} + S^{ij})\Big){\nabla}_{i}u{\nabla}_{j}u - \alpha \gamma \Delta u + 2\beta \gamma |\nabla u|^{2} -b\gamma\frac{u}{\tau}.
\end{eqnarray*}
\end{prop}
\begin{proof}
Notice that we have the following by the definition of $H_{S}$:
\begin{eqnarray*}
{\Delta}H_{S} &=& \alpha \Delta(\Delta u) - \beta \Delta(|\nabla u|^{2}) + a\Delta{S} + \frac{b}{\tau}\Delta u.
\end{eqnarray*}
By (\ref{u-differ}) and Lemma \ref{lem-1-com}, we obtain
\begin{eqnarray*}
\frac{\partial H_{S}}{\partial \tau} &=& \alpha \frac{\partial}{\partial \tau} (\Delta u) - \beta \frac{\partial}{\partial \tau}(|\nabla u|^{2}) + a\frac{\partial S}{\partial \tau} + \frac{b}{\tau}\frac{\partial u}{\partial \tau} - b\frac{u}{\tau^{2}}- d\frac{n}{\tau^{2}} \\
&=& \alpha \Big(\Delta (\Delta u) - \Delta(|\nabla u|^{2}) -c\Delta S - 2{S}^{ij}{\nabla}_{i}{\nabla}_{j}u - (2{\nabla}^{i}S_{i\ell} - {\nabla}_{\ell}S){\nabla}^{\ell}u - \gamma \Delta u \Big) \\
&-& \beta \Big(\Delta(|\nabla u|^{2}) -2|\nabla \nabla u|^{2} -2{\nabla}^{i}(|\nabla u|^{2}){\nabla}_{i}u -2c{\nabla}^{i}S{\nabla}_{i}u \\
&-& 2(R^{ij} + S^{ij}){\nabla}_{i}u{\nabla}_{j}u - 2 \gamma|\nabla u|^{2} \Big) + a\frac{\partial S}{\partial \tau} + \frac{b}{\tau}\Big(\Delta u - |\nabla u|^{2} - cS - \gamma u \Big) \\
&-& b\frac{u}{\tau^{2}} - d\frac{n}{\tau^{2}} \\
&=& {\Delta}H_{S} - \alpha \Delta(|\nabla u|^{2}) -\alpha c\Delta S - 2\alpha{S}^{ij}{\nabla}_{i}{\nabla}_{j}u + 2\beta|\nabla \nabla u|^{2} \\
&+& 2{\beta}{\nabla}^{i}(|\nabla u|^{2}){\nabla}_{i}u + 2{\beta}c{\nabla}^{i}S{\nabla}_{i}u + 2\beta (R^{ij} + S^{ij}){\nabla}_{i}u {\nabla}_{j}u - b\frac{|\nabla u|^{2}}{\tau} \\
&-& b\frac{cS}{\tau} - b\frac{u}{\tau^{2}} - d\frac{n}{\tau^{2}} + a\Big(\frac{\partial S}{\partial \tau}- \Delta S   \Big) - {\alpha}\Big(2{\nabla}^{i}S_{i\ell} - {\nabla}_{\ell}S \Big){\nabla}^{\ell}u
\\
&-& \alpha \gamma \Delta u + 2\beta \gamma |\nabla u|^{2} -b\gamma\frac{u}{\tau}.
\end{eqnarray*}
On the other hand, we also have the following by the definition of $H_{S}$:
\begin{eqnarray*}\label{Def-H-22b}
{\nabla}^{i}H_{S} &=& \alpha \nabla^{i}(\Delta u) - \beta \nabla^{i}(|\nabla u|^{2}) + a\nabla^{i}{S} + \frac{b}{\tau}\nabla^{i}u
\end{eqnarray*}
This and (\ref{r-formula}) imply
\begin{eqnarray*}
-2{\nabla}^{i}H_{S}{\nabla}_{i}u &=& -2{\alpha}{\nabla}^{i}(\Delta u){\nabla}_{i}u + 2\beta{\nabla}^{i}(|\nabla u|^{2}){\nabla}_{i}u - 2a{\nabla}^{i}S{\nabla}_{i}u - \frac{2b}{\tau}|\nabla u|^{2} \\
&=& -\alpha {\Delta}(|\nabla u|^{2}) + 2\alpha{R}_{ij}{\nabla}^{i}u{\nabla}^{j}u + 2 \alpha|\nabla \nabla u|^{2} + 2\beta{\nabla}^{i}(|\nabla u|^{2}){\nabla}_{i}u \\
&-& 2a{\nabla}^{i}S{\nabla}_{i}u - \frac{2b}{\tau}|\nabla u|^{2}.
\end{eqnarray*}
Equivalently, we have
\begin{eqnarray*}
-\alpha {\Delta}(|\nabla u|^{2}) &=& -2{\nabla}^{i}H_{S}{\nabla}_{i}u - 2\alpha{R}_{ij}{\nabla}^{i}u{\nabla}^{j}u - 2 \alpha|\nabla \nabla u|^{2} - 2\beta{\nabla}^{i}(|\nabla u|^{2}){\nabla}_{i}u \\
&+& 2a{\nabla}^{i}S{\nabla}_{i}u + \frac{2b}{\tau}|\nabla u|^{2}.
\end{eqnarray*}
The desired result now follows from the above equation on $\frac{\partial H_{S}}{\partial \tau}$ and this equation.
\end{proof}
Let us introduce
\begin{defn}\label{Def-HJ-BA}
Let $g(t)$ evolve by (\ref{GF}) and let $S$ be the trace of $S_{ij}$. Let $X=X^{i}\frac{\partial}{\partial x^{i}}$ be a vector field on $M$. We define
\begin{eqnarray*}
{\mathcal E}_{(a, c, \alpha, \beta)}(S_{ij}, X) &=& a\frac{\partial S}{\partial \tau} - \Big(a + \alpha c \Big) \Delta S + \frac{ \alpha^{2}}{2(\alpha -\beta)}|S_{ij}|^{2} - \alpha \Big(2{\nabla}^{i}S_{i\ell} - {\nabla}_{\ell}S \Big){X}^{\ell} \\
&-& 2\Big(\alpha R^{ij} - \beta(R^{ij} + S^{ij})\Big){X}_{i}{X}_{j},
\end{eqnarray*}
where $a, c, \alpha$, $\beta$ are constants and $\alpha \not= \beta$.
\end{defn}
Then we get
\begin{thm}\label{T-1}
Suppose that $\alpha \not=0$ and $\alpha \not=\beta$. Then, the evolution equation in Proposition \ref{lem-key-1} can be rewritten as follows:
\begin{eqnarray*}\label{Def-H}
\frac{\partial H_{S}}{\partial \tau} &=& \Delta{H_{S}} -2{\nabla}^{i}H_{S}{\nabla}_{i}u - 2(\alpha - \beta)\Big|\nabla_{i}\nabla_{j}u + \frac{\alpha}{2(\alpha-\beta)}S_{ij}-\frac{\lambda}{2\tau}g_{ij} \Big|^{2} \\
&+& 2(a+\beta c){\nabla}^{i}u{\nabla}_{i}S - \frac{2(\alpha -\beta)}{\alpha}\frac{\lambda}{\tau}H_{S} + (\alpha -\beta)\frac{n{\lambda}^{2}}{2\tau^{2}} + \Big(b - \frac{2(\alpha-\beta)\lambda {\beta}}{\alpha} \Big)\frac{|\nabla u|^{2}}{\tau} \\
&+& \Big(\frac{2(\alpha-\beta)}{\alpha}a \lambda - \alpha{\lambda} - bc \Big)\frac{S}{\tau} + \Big(\frac{2(\alpha-\beta)\lambda}{\alpha}-1 \Big)\frac{b}{\tau^{2}}u + \Big(\frac{2(\alpha-\beta)\lambda}{\alpha}-1 \Big)\frac{d}{\tau^{2}}n \\
&-& \alpha \gamma \Delta u + 2\beta \gamma |\nabla u|^{2} -b\gamma\frac{u}{\tau} + {\mathcal E}_{(a,c, \alpha, \beta)}(S_{ij}, {\nabla}u),
\end{eqnarray*}
where $\lambda$ is a constant.
\end{thm}
\begin{proof}
Notice that a direct computation implies
\begin{eqnarray*}
-2(\alpha-\beta)\Big|\nabla_{i}\nabla_{j}u + \frac{\alpha}{2(\alpha-\beta)}S_{ij}-\frac{\lambda}{2\tau}g_{ij} \Big|^{2} &=& -2(\alpha-\beta)|\nabla \nabla u|^{2} - 2{\alpha}S^{ij}{\nabla}_{i}{\nabla}_{j}u \\
&+& 2(\alpha-\beta)\frac{\lambda}{\tau}\Delta u - \frac{\alpha^{2}}{2(\alpha-\beta)}|S_{ij}|^{2} \\
&+& \frac{\lambda}{\tau}\alpha S - (\alpha-\beta)\frac{\lambda^{2}}{2{\tau}^{2}}n.
\end{eqnarray*}
By this and Proposition \ref{lem-key-1}, we obtain
\begin{eqnarray*}\label{Def-H}
\frac{\partial H_{S}}{\partial \tau} &=& \Delta{H_{S}} -2{\nabla}^{i}H_{S}{\nabla}_{i}u - 2(\alpha - \beta)\Big|\nabla_{i}\nabla_{j}u + \frac{\alpha}{2(\alpha-\beta)}S_{ij}-\frac{\lambda}{2\tau}g_{ij} \Big|^{2} \\
&-& 2\Big(\alpha R^{ij} - \beta(R^{ij} + S^{ij})\Big){\nabla}_{i}u{\nabla}_{j}u + 2(a+\beta c){\nabla}^{i}S{\nabla}_{i}u + (\alpha-\beta)\frac{\lambda^{2}}{2{\tau}^{2}}n \\
&-& 2(\alpha-\beta)\frac{\lambda}{\tau}\Delta u - \frac{\lambda}{\tau}\alpha S  + \frac{b}{\tau}|\nabla u|^{2} + a\frac{\partial S}{\partial \tau} - \Big(a + \alpha c \Big) \Delta S - \alpha \Big(2{\nabla}^{i}S_{i\ell} - {\nabla}_{\ell}S \Big){\nabla}^{\ell}u \\
&+& \frac{\alpha^{2}}{2(\alpha-\beta)}|S_{ij}|^{2} - \frac{b}{\tau}cS - \frac{b}{\tau^{2}}u - d\frac{n}{\tau^{2}} - \alpha \gamma \Delta u + 2\beta \gamma |\nabla u|^{2} -b\gamma\frac{u}{\tau}.
\end{eqnarray*}
The desired result now follows from the above equation, Definition \ref{Def-HJ-BA} and the following which follows from the definition of $H_{S}$:
\begin{eqnarray*}
- \frac{2(\alpha - \beta)}{\alpha}\frac{\lambda}{\tau}H_{S} &=& - 2(\alpha-\beta)\frac{\lambda}{\tau} \Delta u + \frac{2(\alpha - \beta)}{\alpha}\lambda {\beta}\frac{|\nabla u|^{2}}{\tau} - \frac{2(\alpha - \beta)}{\alpha}\lambda {a}\frac{S}{\tau} \\
&-& \frac{2(\alpha - \beta)}{\alpha}\lambda {b}\frac{u}{\tau^{2}} - \frac{2(\alpha - \beta)}{\alpha}\lambda \frac{dn}{\tau^{2}}.
\end{eqnarray*}
\end{proof}

\subsection{Case of $v = - \log f -w(\tau)$}

As in Subsection \ref{sub-u}, let $f$ be a positive solution of the following nonlinear backward heat equation with potential term $- cS$:
\begin{eqnarray*}
\frac{\partial f}{\partial t} = -{\Delta}_{}f + \gamma f\log f- cSf.
\end{eqnarray*}
In what follows, let
$$
v := - \log f -w(\tau),
$$
where $w(\tau)$ is any $C^{\infty}$ function on $\tau =T-t$.
%By a direct computation, we see that $v$ satisfies
%\begin{eqnarray*}\label{v-tau-eq}
%\frac{\partial v}{\partial \tau} = \Delta v - |\nabla v|^{2} - cS -\frac{n}{2\tau} - \gamma \Big( v + \frac{n}{2} \log (4{\pi}\tau) \Big).
%\end{eqnarray*}
Then we have the following:
\begin{thm}\label{T-2}
Let $g(t)$ be a solution to the geometric flow (\ref{GF}). Let
\begin{eqnarray*}
P_{S} = \alpha \Delta v - \beta |\nabla v|^{2} + a{S} + b\frac{v}{\tau}+d\frac{n}{\tau},
\end{eqnarray*}
where $\alpha, \beta, a, b$ and $d$ are constants,  $\alpha \not=0$ and $\alpha \not= \beta$. Then $P_{S}$ satisfies
\begin{eqnarray*}
\frac{\partial P_{S}}{\partial \tau} &=& \Delta{P_{S}} -2{\nabla}^{i}P_{S}{\nabla}_{i}v - 2(\alpha - \beta)\Big|\nabla_{i}\nabla_{j}v + \frac{\alpha}{2(\alpha-\beta)}S_{ij}-\frac{\lambda}{2\tau}g_{ij} \Big|^{2} \\
&+& 2(a+\beta c){\nabla}^{i}v{\nabla}_{i}S - \frac{2(\alpha -\beta)}{\alpha}\frac{\lambda}{\tau}P_{S} + (\alpha -\beta)\frac{n{\lambda}^{2}}{2\tau^{2}} \\
&+& \Big(b - \frac{2(\alpha-\beta)\lambda {\beta}}{\alpha} \Big)\frac{|\nabla v|^{2}}{\tau} + \Big(\frac{2(\alpha-\beta)}{\alpha}a \lambda - \alpha{\lambda} - bc \Big)\frac{S}{\tau} \\
&+& \Big(\frac{2(\alpha-\beta)\lambda}{\alpha}-1 \Big)\frac{b}{\tau^{2}}v + \Big(\frac{2(\alpha-\beta)\lambda}{\alpha}-1 \Big)\frac{d}{\tau^{2}}n - \alpha \gamma \Delta v + 2\beta \gamma |\nabla v|^{2} \\
&-& b\gamma\frac{1}{\tau}\Big(v + w(\tau) \Big) + {\mathcal E}_{(a,c, \alpha, \beta)}(S_{ij}, {\nabla}v) - \frac{b}{\tau}\frac{\partial w(\tau)}{\partial \tau},
\end{eqnarray*}
where $\lambda$ is a constant.
\end{thm}
\begin{proof}
A similar computation with Theorem \ref{T-1} enables us to prove this result. In fact, notice that we have $v = u - w(\tau)$. Therefore we get $\nabla u = \nabla v$ and $\Delta u = \Delta v$. We also have
\begin{eqnarray*}\label{PS-HS}
P_{S} =H_{S} - \frac{b}{\tau}w(\tau).
\end{eqnarray*}
Then Theorem \ref{T-1} and direct computations imply
\begin{eqnarray*}
\frac{\partial P_{S}}{\partial \tau} &=& \frac{\partial H_{S}}{\partial \tau} + \frac{b}{{\tau}^{2}}w(\tau) - \frac{b}{{\tau}}\frac{\partial w(\tau)}{\partial \tau} \\
&=& \Delta{H_{S}} -2{\nabla}^{i}H_{S}{\nabla}_{i}u - 2(\alpha - \beta)\Big|\nabla_{i}\nabla_{j}u + \frac{\alpha}{2(\alpha-\beta)}S_{ij}-\frac{\lambda}{2\tau}g_{ij} \Big|^{2} \\
&+& 2(a+\beta c){\nabla}^{i}u{\nabla}_{i}S - \frac{2(\alpha -\beta)}{\alpha}\frac{\lambda}{\tau}H_{S} + (\alpha -\beta)\frac{n{\lambda}^{2}}{2\tau^{2}} + \Big(b - \frac{2(\alpha-\beta)\lambda {\beta}}{\alpha} \Big)\frac{|\nabla u|^{2}}{\tau} \\
&+& \Big(\frac{2(\alpha-\beta)}{\alpha}a \lambda - \alpha{\lambda} - bc \Big)\frac{S}{\tau} + \Big(\frac{2(\alpha-\beta)\lambda}{\alpha}-1 \Big)\frac{b}{\tau^{2}}u + \Big(\frac{2(\alpha-\beta)\lambda}{\alpha}-1 \Big)\frac{d}{\tau^{2}}n \\
&-& \alpha \gamma \Delta u + 2\beta \gamma |\nabla u|^{2} -b\gamma\frac{u}{\tau} + {\mathcal E}_{(a,c, \alpha, \beta)}(S_{ij}, {\nabla}u) + \frac{b}{{\tau}^{2}}w(\tau) - \frac{b}{{\tau}}\frac{\partial w(\tau)}{\partial \tau}.
\end{eqnarray*}
Since we have $\nabla u = \nabla v$ and $\Delta u = \Delta v$, this implies
\begin{eqnarray*}
\frac{\partial P_{S}}{\partial \tau}
&=& \Delta{P_{S}} -2{\nabla}^{i}P_{S}{\nabla}_{i}v - 2(\alpha - \beta)\Big|\nabla_{i}\nabla_{j}v + \frac{\alpha}{2(\alpha-\beta)}S_{ij}-\frac{\lambda}{2\tau}g_{ij} \Big|^{2} \\
&+& 2(a+\beta c){\nabla}^{i}v{\nabla}_{i}S - \frac{2(\alpha -\beta)}{\alpha}\frac{\lambda}{\tau}\Big(P_{S} + \frac{b}{\tau}w(\tau) \Big) + (\alpha -\beta)\frac{n{\lambda}^{2}}{2\tau^{2}} \\
&+& \Big(b - \frac{2(\alpha-\beta)\lambda {\beta}}{\alpha} \Big)\frac{|\nabla v|^{2}}{\tau} + \Big(\frac{2(\alpha-\beta)}{\alpha}a \lambda - \alpha{\lambda} - bc \Big)\frac{S}{\tau} \\
&+& \Big(\frac{2(\alpha-\beta)\lambda}{\alpha}-1 \Big)\frac{b}{\tau^{2}}\Big(v + w(\tau) \Big) + \Big(\frac{2(\alpha-\beta)\lambda}{\alpha}-1 \Big)\frac{d}{\tau^{2}}n \\
&-& \alpha \gamma \Delta v + 2\beta \gamma |\nabla v|^{2} -b\gamma\frac{1}{\tau}\Big(v + w(\tau) \Big) + {\mathcal E}_{(a,c, \alpha, \beta)}(S_{ij}, {\nabla}v) + \frac{b}{{\tau}^{2}}w(\tau) - \frac{b}{{\tau}}\frac{\partial w(\tau)}{\partial \tau} \\
&=& \Delta{P_{S}} -2{\nabla}^{i}P_{S}{\nabla}_{i}v - 2(\alpha - \beta)\Big|\nabla_{i}\nabla_{j}v + \frac{\alpha}{2(\alpha-\beta)}S_{ij}-\frac{\lambda}{2\tau}g_{ij} \Big|^{2} \\
&+& 2(a+\beta c){\nabla}^{i}v{\nabla}_{i}S - \frac{2(\alpha -\beta)}{\alpha}\frac{\lambda}{\tau}P_{S} + (\alpha -\beta)\frac{n{\lambda}^{2}}{2\tau^{2}} \\
&+& \Big(b - \frac{2(\alpha-\beta)\lambda {\beta}}{\alpha} \Big)\frac{|\nabla v|^{2}}{\tau} + \Big(\frac{2(\alpha-\beta)}{\alpha}a \lambda - \alpha{\lambda} - bc \Big)\frac{S}{\tau} \\
&+& \Big(\frac{2(\alpha-\beta)\lambda}{\alpha}-1 \Big)\frac{b}{\tau^{2}}v + \Big(\frac{2(\alpha-\beta)\lambda}{\alpha}-1 \Big)\frac{d}{\tau^{2}}n - \alpha \gamma \Delta v + 2\beta \gamma |\nabla v|^{2} \\
&-& b\gamma\frac{1}{\tau}\Big(v +  w(\tau) \Big) + {\mathcal E}_{(a,c, \alpha, \beta)}(S_{ij}, {\nabla}v) - \frac{b}{\tau}\frac{\partial w(\tau)}{\partial \tau}.
\end{eqnarray*}
Hence we obtained the desired result.
\end{proof}

As s corollary of Theorem \ref{T-2}, we get
\begin{cor}\label{PS-cor}
Let $g(t)$ be a solution to the geometric flow (\ref{GF}) and $f$ be a positive solution of the following :
\begin{eqnarray*}
\frac{\partial f}{\partial t} = -{\Delta}_{}f + \gamma f\log f +aSf.
\end{eqnarray*}
Let $v = - \log f -w(\tau)$, $\tau = T-t$ and
\begin{eqnarray*}
P_{S}= 2\Delta v - |\nabla v|^{2} + a{S} + \frac{dn}{\tau}.
\end{eqnarray*}
Then, the following holds:
\begin{eqnarray*}
\frac{\partial P_{S}}{\partial \tau} &\leq& \Delta{P_{S}} -2{\nabla}^{i}P_{S}{\nabla}_{i}v - \Big(\frac{2}{\tau} + 2\gamma \Big)P_{S} - \frac{2}{n}(1-\gamma)\Big( \Delta v + S -\frac{n}{\tau} \Big)^{2} \\
&-& \frac{n}{2}\gamma \Big(\Delta v + S - \frac{n}{\tau}- \frac{n}{2} \Big)^{2} + 2S \Big(\frac{a-2}{\tau}+(a-1) \gamma \Big) + \frac{n}{\tau}\Big(\frac{1}{\tau}(d+2) \\
&+& 2\gamma (d+1) \Big) + \frac{n}{2}\gamma - 2\frac{|\nabla v|^{2}}{\tau} + {\mathcal E}_{(a, -a, 2, 1)}(S_{ij}, {\nabla}v).
\end{eqnarray*}
In particular, if $0 \leq \gamma \leq 1$, then
\begin{eqnarray*}
\frac{\partial P_{S}}{\partial \tau} &\leq& \Delta{P_{S}} -2{\nabla}^{i}P_{S}{\nabla}_{i}v - \Big(\frac{2}{\tau} + 2\gamma \Big)P_{S} + 2S \Big(\frac{a-2}{\tau}+(a-1) \gamma \Big) \\
&+& \frac{n}{\tau}\Big(\frac{1}{\tau}(d+2) + 2\gamma (d+1) \Big) + \frac{n}{2}\gamma - 2\frac{|\nabla v|^{2}}{\tau} + {\mathcal E}_{(a, -a, 2, 1)}(S_{ij}, {\nabla}v).
\end{eqnarray*}
\end{cor}

\begin{proof}
By Theorem \ref{T-2} in the case where $\alpha = 2$, $\beta = 1$, $\lambda = 2$, $a =- c$, $\lambda = 2$, $b = 0$, we obtain
\begin{eqnarray*}
\frac{\partial P_{S}}{\partial \tau} &=& \Delta{P_{S}} -2{\nabla}^{i}P_{S}{\nabla}_{i}v - 2\Big|\nabla_{i}\nabla_{j}v + S_{ij}-\frac{1}{\tau}g_{ij} \Big|^{2} - \frac{2}{\tau}P_{S} - \frac{2}{\tau}{|\nabla v|^{2}} \\
&+& 2(a-2) \frac{S}{\tau}- 2\gamma(\Delta v - |\nabla v|^{2}) + \frac{n}{\tau^{2}}(d+2) + {\mathcal E}_{(a, -a, 2,1)}(S_{ij}, {\nabla}v).
\end{eqnarray*}
By the definition of $P_S$, we have
\begin{eqnarray*}
-2\gamma(\Delta v - |\nabla v|^{2}) = -2\gamma{P}_{S} + 2 \gamma\Big(\Delta v + S - \frac{n}{\tau} \Big) + 2\gamma(a-1)S  + \frac{2n}{\tau}(d+1)\gamma.
\end{eqnarray*}
On the other hand, we also have
\begin{eqnarray*}
\Big|\nabla_{i}\nabla_{j}v + S_{ij}-\frac{1}{\tau}g_{ij} \Big|^{2} \geq \frac{1}{n} \Big(\Delta v +S - \frac{n}{\tau} \Big)^{2}.
\end{eqnarray*}
Therefore, we obain
\begin{eqnarray*}
\frac{\partial P_{S}}{\partial \tau} &\leq& \Delta{P_{S}} -2{\nabla}^{i}P_{S}{\nabla}_{i}v - \Big(\frac{2}{\tau} + 2\gamma \Big) P_{S}  - \frac{2}{\tau}{|\nabla v|^{2}} + {\mathcal E}_{(a, -a, 2,1)}(S_{ij}, {\nabla}v) \\
&+&  2(a-2) \frac{S}{\tau} - \frac{2}{n} \Big(\Delta v + S - \frac{n}{\tau} \Big)^{2} + 2\gamma \Big(\Delta v + S - \frac{n}{\tau} \Big) +2\gamma \Big(a-1 \Big)S \\
&+& \frac{n}{\tau}\Big(\frac{1}{\tau}(d+2) + 2\gamma(d+1) \Big)
\end{eqnarray*}
On the other hand, we get the following by a direct computation.
\begin{eqnarray*}
2\gamma \Big(\Delta v + S - \frac{n}{\tau} \Big) = -\frac{2}{n}\gamma \Big(\Delta v +S - \frac{n}{\tau}-\frac{n}{2} \Big)^{2} + \frac{2}{n} \gamma\Big(\Delta v + S - \frac{n}{\tau} \Big)^{2} + \frac{2}{n} \gamma.
\end{eqnarray*}
By using this, we get the desired result.
\end{proof}

%\begin{rem}\label{rem-2}
%We notice that Corollary \ref{PS-cor} also follows from Corollary \ref{HS-cor} by using (\ref{PS-HS}) because we considered the case where $b=0$.
%\end{rem}

%%%%%%%%%%%%%%%%%%%%%%%%%%%%%%%%%%%%%%%%%%%%
\section{Proof of Theorem \ref{main-A}} \label{POA}
%%%%%%%%%%%%%%%%%%%%%%%%%%%%%%%%%%%%%%%%%%%%

In this section, we shall prove Theorem \ref{main-A}.

\subsection{The first case}

Suppose that $d \leq -2$ and $0 \leq \gamma \leq 1$. We have the following by Corollary \ref{PS-cor}:
\begin{eqnarray*}
\frac{\partial P_{S}}{\partial \tau} &\leq& \Delta{P_{S}} -2{\nabla}^{i}P_{S}{\nabla}_{i}v  - (\frac{2}{\tau} + 2\gamma)P_{S} + \frac{n}{2}\gamma+ 2S \Big(\frac{a-2}{\tau} + \gamma(a-1) \Big) + {\mathcal E}_{a}(S_{ij}, {\nabla}v),
\end{eqnarray*}
where notice that ${\mathcal E}_{(a, -a, 2,1)}(S_{ij}, {\nabla}v) = {\mathcal E}_{a}(S_{ij}, {\nabla}v)$ holds.
Since we assumed $a=2$, this implies
\begin{eqnarray*}
\frac{\partial P_{S}}{\partial \tau} &\leq& \Delta{P_{S}} -2{\nabla}^{i}P_{S}{\nabla}_{i}v  - (\frac{2}{\tau} + 2\gamma)P_{S} + \frac{n}{2}\gamma+ 2S\gamma + {\mathcal E}_{2}(S_{ij}, {\nabla}v).
\end{eqnarray*}
On the other hand, a direct computation tells us that
$$
2S\gamma = -\frac{2}{n} \gamma \Big(S -\frac{n}{2} \Big)^{2} + \frac{2}{n}\gamma S^{2} + \frac{n}{2}\gamma.
$$
Hence, we get
\begin{eqnarray*}
\frac{\partial P_{S}}{\partial \tau} &\leq& \Delta{P_{S}} -2{\nabla}^{i}P_{S}{\nabla}_{i}v  - (\frac{2}{\tau} + 2\gamma)P_{S} + n\gamma  -\frac{2}{n} \gamma \Big(S -\frac{n}{2} \Big)^{2} + \frac{2}{n}\gamma S^{2} + {\mathcal E}_{2}(S, {\nabla}v) \\
&\leq& \Delta{P_{S}} -2{\nabla}^{i}P_{S}{\nabla}_{i}v  - (\frac{2}{\tau} + 2\gamma)P_{S} + n\gamma + {\mathcal E}_{2}(S_{ij}, {\nabla}v) + \frac{2}{n}\gamma S^{2}.
\end{eqnarray*}
On the other hand, we also have
$$
- (\frac{2}{\tau} + 2\gamma)P_{S}= - (\frac{2}{\tau} + 2\gamma)\Big(P_{S}-\frac{n}{2}\gamma \Big) - \frac{n \gamma}{\tau}- {n}{\gamma}^{2}.
$$
Adding $-\frac{n}{2}\gamma$ to $P_{S}$, we get
\begin{eqnarray*}
\frac{\partial }{\partial \tau}\Big(P_{S}-\frac{n}{2}\gamma \Big)  &\leq& \Delta{\Big( P_{S}-\frac{n}{2}\gamma \Big)} -2{\nabla}^{i}\Big(P_{S}-\frac{n}{2}\gamma \Big){\nabla}_{i}v - \Big( \frac{2}{\tau} + 2\gamma \Big) \Big( P_{S}-\frac{n}{2}\gamma \Big) - \frac{n}{\tau}\gamma \\
&-& n(\gamma^{2} - \gamma) + {\mathcal E}_{2}(S_{ij}, {\nabla}v) + \frac{2}{n}\gamma S^{2} \\
&\leq& \Delta{\Big( P_{S}-\frac{n}{2}\gamma \Big)} -2{\nabla}^{i}\Big(P_{S}-\frac{n}{2}\gamma \Big){\nabla}_{i}v - \Big( \frac{2}{\tau} + 2\gamma \Big) \Big( P_{S}-\frac{n}{2}\gamma \Big) \\
&-& n\gamma(\gamma - 1) + {\mathcal E}_{2}(S_{ij}, {\nabla}v) + \frac{2}{n}\gamma S^{2}.
\end{eqnarray*}
Finally, by taking $\gamma = 1$ and using (\ref{key-cond-2}), we obtain
\begin{eqnarray*}
\frac{\partial }{\partial \tau}\Big(P_{S}-\frac{n}{2} \Big)
&\leq& \Delta{\Big( P_{S}-\frac{n}{2} \Big)} -2{\nabla}^{i}\Big(P_{S}-\frac{n}{2} \Big){\nabla}_{i}v - \Big( \frac{2}{\tau} + 2 \Big) \Big( P_{S}-\frac{n}{2}\Big).
\end{eqnarray*}
Since
$$
P_{S}-\frac{n}{2} < 0
$$
holds for $\tau$ small enough which depends on $d$, the maximal principle implies the desired result.

\subsection{The second case}

Suppose that $d \leq -2$ and $0 \leq \gamma \leq 1$. As in the first case, the following holds:
\begin{eqnarray*}
\frac{\partial P_{S}}{\partial \tau} &\leq& \Delta{P_{S}} -2{\nabla}^{i}P_{S}{\nabla}_{i}v  - (\frac{2}{\tau} + 2\gamma)P_{S} + \frac{n}{2}\gamma+ 2S \Big(\frac{a-2}{\tau} + \gamma(a-1) \Big) + {\mathcal E}_{a}(S_{ij}, {\nabla}v).
\end{eqnarray*}
Since we assumed $a=1$, we have
\begin{eqnarray*}
\frac{\partial P_{S}}{\partial \tau} &\leq& \Delta{P_{S}} -2{\nabla}^{i}P_{S}{\nabla}_{i}v  - (\frac{2}{\tau} + 2\gamma )P_{S} + \frac{n}{2}\gamma - \frac{2S}{\tau} + {\mathcal E}_{1}(S_{ij}, {\nabla}v).
\end{eqnarray*}
Assume (\ref{key-cond-3}) holds. Then this imples
\begin{eqnarray*}
\frac{\partial P_{S}}{\partial \tau} &\leq& \Delta{P_{S}} -2{\nabla}^{i}P_{S}{\nabla}_{i}v  - (\frac{2}{\tau} + 2\gamma)P_{S} + \frac{n}{2}\gamma.
\end{eqnarray*}
By adding $-\frac{n}{4}\gamma$ to $P_{S}$, we get
\begin{eqnarray*}
\frac{\partial }{\partial \tau}\Big( P_{S}-\frac{n}{4}\gamma \Big) &\leq& \Delta\Big({P_{S}}-\frac{n}{4}\gamma  \Big) -2{\nabla}^{i}\Big({P_{S}}-\frac{n}{4}\gamma  \Big){\nabla}_{i}v - \Big(\frac{2}{\tau} + 2\gamma  \Big)\Big( P_{S} -\frac{n}{4}\gamma \Big) - \frac{n\gamma }{2\tau} \\
&-& \frac{n}{2}({\gamma}^{2} - \gamma) \\
&\leq& \Delta\Big({P_{S}}-\frac{n}{4}\gamma  \Big) -2{\nabla}^{i}\Big({P_{S}}-\frac{n}{4}\gamma  \Big){\nabla}_{i}v - \Big(\frac{2}{\tau} + 2\gamma  \Big)\Big( P_{S} -\frac{n}{4}\gamma \Big) \\
&-& \frac{n}{2}\gamma({\gamma} - 1)
\end{eqnarray*}
Finally, by taking $\gamma = 1$,
\begin{eqnarray*}
\frac{\partial }{\partial \tau}\Big( P_{S}-\frac{n}{4} \Big) &\leq& \Delta\Big({P_{S}}-\frac{n}{4} \Big) -2{\nabla}^{i}\Big({P_{S}}-\frac{n}{4} \Big){\nabla}_{i}v - \Big(\frac{2}{\tau} + 2 \Big)\Big( P_{S} -\frac{n}{4}\Big).
\end{eqnarray*}
Notice that
$$
P_{S}-\frac{n}{4} < 0
$$
holds for $\tau$ small enough which depends on $d$. By using the maximal principle, we get the desired result.

\section{Proof of Theorem \ref{main-Aa}}\label{thm-Aa}
%%%%%%%%%%%%%%%%%%%%%%%%%%%%%%%%%%%%%%%%%%%%

%We are now in a position to prove Theorem \ref{main-Aa}. 
Suppose that $0 \leq \gamma \leq 1$. We get the following by Corollary \ref{PS-cor}:
\begin{eqnarray*}
\frac{\partial P_{S}}{\partial \tau} &\leq& \Delta{P_{S}} -2{\nabla}^{i}P_{S}{\nabla}_{i}v  - (\frac{2}{\tau} + 2\gamma)P_{S} + \frac{n}{2}\gamma + 2S \Big(\frac{a-2}{\tau} + (a-1)\gamma \Big) \\
&+& \frac{n}{\tau}\Big(\frac{1}{\tau}(d+2) + 2\gamma(d+1) \Big) + {\mathcal E}_{a}(S_{ij}, {\nabla}v).
\end{eqnarray*}
By taking  $a=1$ and $d = -3$, we get
\begin{eqnarray*}
\frac{\partial P_{S}}{\partial \tau}
&\leq & \Delta{P_{S}} -2{\nabla}^{i}P_{S}{\nabla}_{i}v - \Big(\frac{2}{\tau} + 2\gamma \Big)P_{S} - \frac{2}{\tau}\Big( S + \frac{n}{2\tau} \Big) - \frac{4n}{\tau}\gamma + \frac{n}{2} \gamma+ {\mathcal E}_{1}(S, {\nabla}v) \\
&\leq & \Delta{P_{S}} -2{\nabla}^{i}P_{S}{\nabla}_{i}v - \Big(\frac{2}{\tau} + 2\gamma \Big)P_{S} - \frac{2}{\tau}\Big( S + \frac{n}{2\tau} \Big) + \frac{n}{2} \gamma+ {\mathcal E}_{1}(S_{ij}, {\nabla}v)
\end{eqnarray*}
Asuume that ${\mathcal E}_{1}(S_{ij}, {\nabla}v) \leq 0$ holds. Then we obtain
\begin{eqnarray*}
\frac{\partial P_{S}}{\partial \tau} &\leq & \Delta{P_{S}} -2{\nabla}^{i}P_{S}{\nabla}_{i}v - \Big(\frac{2}{\tau} + 2\gamma \Big)P_{S} - \frac{2}{\tau}\Big( S + \frac{n}{2\tau} \Big) + \frac{n}{2}\gamma.
\end{eqnarray*}
On the other hand, assume that $S \geq -\frac{n}{2t}$ holds for all time $t \in [\frac{T}{2}, T)$. Since we have $\frac{1}{\tau} \geq \frac{1}{t}$, we get
$$
S \geq -\frac{n}{2t} \geq  -\frac{n}{2\tau}.
$$
Therefore, we obtain
\begin{eqnarray*}
\frac{\partial P_{S}}{\partial \tau} &\leq& \Delta{P_{S}} -2{\nabla}^{i}P_{S}{\nabla}_{i}v - \Big(\frac{2}{\tau} + 2\gamma \Big)P_{S}+ \frac{n}{2}\gamma.
\end{eqnarray*}
By adding $-\frac{n}{4}\gamma$ to $P_{S}$, we have
\begin{eqnarray*}
\frac{\partial }{\partial \tau}\Big(P_{S} -\frac{n}{4}\gamma \Big) &\leq& \Delta \Big(P_{S} -\frac{n}{4}\gamma \Big)  -2{\nabla}^{i}\Big(P_{S} -\frac{n}{4}\gamma \Big) {\nabla}_{i}v - \Big(\frac{2}{\tau} + 2\gamma \Big)\Big(P_{S} -\frac{n}{4}\gamma \Big) - \frac{n\gamma}{2\tau} \\
&-& \frac{n}{2}(\gamma^{2} - \gamma) \\
&\leq& \Delta \Big(P_{S} -\frac{n}{4} \Big)  -2{\nabla}^{i}\Big(P_{S} -\frac{n}{4} \Big) {\nabla}_{i}v - \Big(\frac{2}{\tau} + 2 \Big)\Big(P_{S} -\frac{n}{4} \Big)-\frac{n}{2}\gamma(\gamma - 1).
\end{eqnarray*}
Finally, by taking $\gamma = 1$, we obtain
\begin{eqnarray*}
\frac{\partial }{\partial \tau}\Big( P_{S}-\frac{n}{4} \Big) &\leq& \Delta\Big({P_{S}}-\frac{n}{4} \Big) -2{\nabla}^{i}\Big({P_{S}}-\frac{n}{4} \Big){\nabla}_{i}v - \Big(\frac{2}{\tau} + 2 \Big)\Big( P_{S} -\frac{n}{4}\Big).
\end{eqnarray*}
Notice that
$$
P_{S}-\frac{n}{4} < 0
$$
holds for $\tau$ small enough. By using the maximal principle, we get Theorem \ref{main-Aa}. \par

%On the other hand, we remark that a similar result to Theorem \ref{main-Aa} still holds for $v = -\log f-w(\tau)$.
%\begin{thm}\label{v-thm-2}
%Let $(M, g(t))$, $t \in [0, T)$, be a solution to the the geometric flow (\ref{GF}) on a closed oriented smooth $n$-manifold $M$. Let $f$ be a positive solution to the following equation
%\begin{eqnarray*}
%\frac{\partial f}{\partial t} = -{\Delta}f + f\log f + Sf.
%\end{eqnarray*}
%Let $u = -\log f-w(\tau)$, $\tau=T-t$. Suppose that
%\begin{eqnarray}\label{s-t-bound-2}
%{\mathcal E}_{1}(S,X) \leq 0, \ S(t) \geq -\frac{n}{2t}
%\end{eqnarray}
%hold for all vector fields $X$ and all time $t \in [\frac{T}{2}, T)$ for which the flow exists.
%Then for all time $t \in [\frac{T}{2}, T)$,
%\begin{eqnarray*}
%2{\Delta}v - |\nabla v|^{2}+{S} -3 \frac{n}{\tau} \leq \frac{n}{4}.
%\end{eqnarray*}
%\end{thm}
%\begin{proof}
%The proof is similar to that of Theorem \ref{main-Aa}. Use Corollary \ref{PS-cor} instead of Corollary \ref{HS-cor}.
%\end{proof}
%For the Ricci flow case,  Theorem \ref{v-thm-2} in the case where $w(\tau)=\frac{n}{2}\log(4{\pi}\tau)$ is nothing but Theorem 1.10 in \cite{Wu}.

%%%%%%%%%%%%%%%%%%%%%%%%%%%%%%%%%%%%%%%%%%%%
\section{Proof of Theorem \ref{main-B}}\label{sec-B}
%%%%%%%%%%%%%%%%%%%%%%%%%%%%%%%%%%%%%%%%%%%%

By Proposition \ref{lem-key-1} in the case where $\alpha = 0$, $\beta = -1$, $a = 0$, $c=0$, $b = -1$ and $d= 0$, we get
$$
H_{S} = |\nabla u|^{2} - \frac{u}{\tau}
$$
and
\begin{eqnarray*}
\frac{\partial H_{S}}{\partial \tau} &=& \Delta{H_{S}} -2{\nabla}^{i}H_{S}{\nabla}_{i}u - \frac{1}{\tau}H_{S} - 2|\nabla \nabla u|^{2} - 2\Big(R^{ij} + S^{ij} \Big){\nabla}_{i}u{\nabla}_{j}u \\
&-&2\gamma|\nabla u|^{2} + \frac{u}{\tau}\gamma \\
&=& \Delta{H_{S}} -2{\nabla}^{i}H_{S}{\nabla}_{i}u -\Big(\frac{1}{\tau} + \gamma \Big)H_{S} - 2|\nabla \nabla u|^{2} - 2\Big(R^{ij} + S^{ij} \Big){\nabla}_{i}u{\nabla}_{j}u \\
&-& \gamma|\nabla u|^{2} \\
&\leq& \Delta{H_{S}} -2{\nabla}^{i}H_{S}{\nabla}_{i}u -\Big(\frac{1}{\tau} + \gamma \Big)H_{S} - 2\Big(R^{ij} + S^{ij} \Big){\nabla}_{i}u{\nabla}_{j}u - \gamma|\nabla u|^{2}.
\end{eqnarray*}
Assume that $R_{ij}(t) + S_{ij}(t) \geq -Ag_{ij}$, where $0 \leq A \leq \frac{1}{2}\gamma$. Then we have
\begin{eqnarray*}
- 2\Big(R^{ij} + S^{ij} \Big){\nabla}_{i}u{\nabla}_{j}u - \gamma|\nabla u|^{2} \leq (2A -\gamma) |\nabla u|^{2} \leq 0.
\end{eqnarray*}
Therefore we obtain
\begin{eqnarray*}
\frac{\partial H_{S}}{\partial \tau}
&\leq& \Delta{H_{S}} -2{\nabla}^{i}H_{S}{\nabla}_{i}u -\Big(\frac{1}{\tau} + \gamma \Big)H_{S}.
\end{eqnarray*}
If $\tau$ is small, then $H_{S} < 0$. By applying the maximal principle to the above inequality, Theorem \ref{main-B} follows. \\
%\begin{rem}
%We remark that Theorem \ref{T-1} cannot be used to prove Theorem \ref{main-B} because we considered the case where $\alpha =0$ in the proof. Instead of Theorem \ref{T-1}, we used Proposition \ref{lem-key-1} to prove Theorem \ref{main-B}. Similarly, we use Proposition \ref{lem-key-1} to prove Theorem \ref{main-C}.
%\end{rem}
On the other hand, it is also natural to consider the following equation:
\begin{eqnarray*}
\frac{\partial f}{\partial \tau} = {\Delta}_{}f - \gamma(\tau) f\log f -aSf,
\end{eqnarray*}
where $\gamma(\tau) $ is any $C^{\infty}$ function of $\tau - T-t$.
Let $u=-\log f$. Then $u$ satisfies
\begin{eqnarray}\label{u-differ-C}
\frac{\partial u}{\partial \tau} = \Delta u - |\nabla u|^{2} +aS - \gamma(\tau) u.
\end{eqnarray}
Then, computations which are similar to Proposition \ref{lem-key-1} tell us that the following holds:
\begin{prop}\label{lem-key-1C}
Let $g(t)$ be a solution to the geometric flow (\ref{GF}) and $u$ satisfies (\ref{u-differ-C}). Let
\begin{eqnarray*}
H_{S} = \alpha \Delta u - \beta |\nabla u|^{2} + a{S} + b\frac{u}{\tau} + d\frac{n}{\tau},
\end{eqnarray*}
where $\alpha, \beta, a, b$ and $d$ are constants. Then the following holds:
\begin{eqnarray*}
\frac{\partial H_{S}}{\partial \tau}
&=& {\Delta}H_{S} - 2{\nabla}^{i}H_{S}{\nabla}_{i}u + 2(a+\beta c){\nabla}^{i}S{\nabla}_{i}u-2(\alpha-\beta)|\nabla \nabla u|^{2} \\
&-&2{\alpha}S^{ij}{\nabla}_{i}{\nabla}_{j}u  + \frac{b}{\tau}|\nabla u|^{2} - \frac{b}{\tau}cS - \frac{b}{\tau^{2}}u - d\frac{n}{\tau^{2}} \\
&+& a\frac{\partial S}{\partial \tau} - \Big(a - \alpha a \Big) \Delta S - \alpha \Big(2{\nabla}^{i}S_{i\ell} - {\nabla}_{\ell}S \Big){\nabla}^{\ell}u \\
&-& 2\Big(\alpha R^{ij} - \beta(R^{ij} + S^{ij})\Big){\nabla}_{i}u{\nabla}_{j}u - \alpha \gamma(\tau) \Delta u + 2\beta \gamma(\tau) |\nabla u|^{2} -b\gamma(\tau)\frac{u}{\tau}
\end{eqnarray*}
holds.
\end{prop}
By using this, we are able to prove a generalization of Theorem \ref{main-B} as follows:
\begin{thm}\label{main-Bbb}
Suppose that $g(t)$, $t \in [0, T)$, evolves by the geometric flow (\ref{GF}) on a closed oriented smooth $n$-manifold $M$ and Let $0<f<1$ be a positive solution to
\begin{eqnarray*}
\frac{\partial f}{\partial \tau} = {\Delta}_{}f -\gamma(\tau) f\log f,
\end{eqnarray*}
where $\gamma(\tau) > 0$ is any $C^{\infty}$ function on $\tau=T-t$. Suppose that there is a constant $A$ such that
\begin{eqnarray*}
R_{ij}(t) + S_{ij}(t) \geq -Ag_{ij}(t),
\end{eqnarray*}
where $0 \leq A \leq \frac{1}{2}\gamma (\tau)$ and Let $u = -\log f$. Then for all time $t \in [0, T)$,
\begin{eqnarray*}
|\nabla u|^{2} - \frac{u}{\tau} \leq 0
\end{eqnarray*}
holds.
\end{thm}
The proof is similar to that of Theorem \ref{main-B}. Instead of Proposition \ref{lem-key-1}, use Proposition \ref{lem-key-1C}. 

%\begin{proof}
%By Proposition \ref{lem-key-1C} in the case where $\alpha = 0$, $\beta = -1$, $a = 0$, $c=0$, $b = -1$ and $d= 0$, we get
%$$
%H_{S} = |\nabla u|^{2} - \frac{u}{\tau}
%$$
%and
%\begin{eqnarray*}
%\frac{\partial H_{S}}{\partial \tau} &=& \Delta{H_{S}} -2{\nabla}^{i}H_{S}{\nabla}_{i}u - \frac{1}{\tau}H_{S} - 2|\nabla \nabla u|^{2} - 2\Big(R^{ij} + S^{ij} \Big){\nabla}_{i}u{\nabla}_{j}u \\
%&-&2\gamma(\tau)|\nabla u|^{2} + \frac{u}{\tau}\gamma(\tau) \\
%&=& \Delta{H_{S}} -2{\nabla}^{i}H_{S}{\nabla}_{i}u -\Big(\frac{1}{\tau} + \gamma(\tau) \Big)H_{S} - 2|\nabla \nabla u|^{2} - 2\Big(R^{ij} + S^{ij} \Big){\nabla}_{i}u{\nabla}_{j}u \\
%&-& \gamma(\tau)|\nabla u|^{2} \\
%&\leq& \Delta{H_{S}} -2{\nabla}^{i}H_{S}{\nabla}_{i}u -\Big(\frac{1}{\tau} + \gamma(\tau) \Big)H_{S} - 2\Big(R^{ij} + S^{ij} \Big){\nabla}_{i}u{\nabla}_{j}u - \gamma(\tau)|\nabla u|^{2}.
%\end{eqnarray*}
%Assume that $R_{ij}(t) + S_{ij}(t) \geq -Ag_{ij}(t)$, where $0 \leq A \leq \frac{1}{2}\gamma(\tau)$. Then we have
%\begin{eqnarray*}
%- 2\Big(R^{ij} + S^{ij} \Big){\nabla}_{i}u{\nabla}_{j}u - \gamma(\tau)|\nabla u|^{2} \leq (2A -\gamma(\tau)) |\nabla u|^{2} \leq 0.
%\end{eqnarray*}
%Therefore we obtain
%\begin{eqnarray*}
%\frac{\partial H_{S}}{\partial \tau}
%&\leq& \Delta{H_{S}} -2{\nabla}^{i}H_{S}{\nabla}_{i}u -\Big(\frac{1}{\tau} + \gamma(\tau) \Big)H_{S}.
%\end{eqnarray*}
%If $\tau$ is small, then $H_{S} < 0$. By applying the maximal principle to the above inequality, we get the claim.
%\end{proof}
%%%%%%%%%%%%%%%%%%%%%%%%%%%%%%%%%%%%%%%%%%%%
\section{Proof of Theorem \ref{main-C}}\label{section-C}
%%%%%%%%%%%%%%%%%%%%%%%%%%%%%%%%%%%%%%%%%%%%

By taking $\alpha = 0$, $\beta = -1$, $\gamma = a = c =0$, $b = -1$ and $d= 0$ in Proposition \ref{lem-key-1}, we have
$$
H_{S} = |\nabla u|^{2} - \frac{u}{\tau}
$$
and
\begin{eqnarray*}
\frac{\partial H_{S}}{\partial \tau}
&=& {\Delta}H_{S} - 2{\nabla}^{i}{H}_{S}{\nabla}_{i}u - \frac{1}{\tau}H_{S} - 2|\nabla \nabla u|^{2}  -2(R^{ij} + S^{ij}){\nabla}_{i}u{\nabla}_{j}u \\
&\leq &{\Delta}H_{S} - 2{\nabla}^{i}{H}_{S}{\nabla}_{i}u - \frac{1}{\tau}H_{S} -2(R^{ij} + S^{ij}){\nabla}_{i}u{\nabla}_{j}u.
\end{eqnarray*}
Assume that $R_{ij}(t) + S_{ij}(t) \geq 0$. Then we have
\begin{eqnarray*}
- 2\Big(R^{ij} + S^{ij} \Big){\nabla}_{i}u{\nabla}_{j}u \leq 0.
\end{eqnarray*}
Therefore, the following holds:
\begin{eqnarray*}
\frac{\partial H_{S}}{\partial \tau} &\leq& {\Delta}H_{S} - 2{\nabla}^{i}{H}_{S}{\nabla}_{i}u - \frac{1}{\tau}H_{S}.
\end{eqnarray*}
If $\tau$ is small, then $H_{S} < 0$.  Then the maximal principle implies the desired result. \par

%%%%%%%%%%%%%%%%%%%%%%%%%%%%%%%%%%%%%%%%%%%%
\section{Proof of Theorem \ref{main-E}}\label{section-E}
%%%%%%%%%%%%%%%%%%%%%%%%%%%%%%%%%%%%%%%%%%%%

Let $P_{S} := 2 \Delta v - |\nabla v|^{2} + {S} -2\frac{n}{\tau}$. By Corollary \ref{PS-cor} in the case where $\gamma =0$, $a=1$ and $d =-2$, we obtain
\begin{eqnarray*}
\frac{\partial P_{S}}{\partial \tau} &\leq& \Delta{P_{S}} -2{\nabla}^{i}P_{S}{\nabla}_{i}v  - \frac{2}{\tau}P_{S} -\frac{2}{n}\Big(\Delta v + S - \frac{n}{\tau} \Big)^{2} \\
&-& \frac{2}{\tau}{|\nabla v|^{2}} - \frac{2S}{\tau} +  {\mathcal E}_{1}(S_{ij}, {\nabla}v) \\
&=& \Delta{P_{S}} -2{\nabla}^{i}P_{S}{\nabla}_{i}v -\frac{2}{\tau} \Big(P_{S} + |\nabla v|^{2} + S \Big) \\
&-& \frac{2}{n}\Big(\Delta v + S - \frac{n}{\tau} \Big)^{2}+  {\mathcal E}_{1}(S_{ij}, {\nabla}v).
\end{eqnarray*}
where notice that ${\mathcal E}_{(1, -1, 2, 1)}(S, {\nabla}v) = {\mathcal E}_{1}(S, {\nabla}v)$. Since we also have
$$
P_{S} + |\nabla v|^{2} + S = 2 \Big(\Delta v + S - \frac{n}{\tau} \Big),
$$
we obtain
\begin{eqnarray*}
\frac{\partial P_{S}}{\partial \tau} &\leq& \Delta{P_{S}} -2{\nabla}^{i}P_{S}{\nabla}_{i}v -\frac{2}{\tau} \Big(P_{S} + |\nabla v|^{2} + S \Big) \\
&-& \frac{1}{2n}\Big(P_{S} + |\nabla v|^{2} + S \Big)^{2}+  {\mathcal E}_{1}(S_{ij}, {\nabla}v) \\
&=& \Delta{P_{S}} -2{\nabla}^{i}P_{S}{\nabla}_{i}v -\frac{1}{2n} \Big(P_{S} + |\nabla v|^{2} + S + \frac{2n}{\tau} \Big)^{2} + \frac{2n}{\tau^{2}} \\
&+& {\mathcal E}_{1}(S_{ij}, {\nabla}v) \\
&\leq& \Delta{P_{S}} -2{\nabla}^{i}P_{S}{\nabla}_{i}v + \frac{2n}{\tau^{2}},
\end{eqnarray*}
where we used ${\mathcal E}_{1}(S_{ij}, {\nabla}u) \leq 0$. By adding $\frac{2n}{\tau}$ to $P_{S}$, we get
\begin{eqnarray*}
\frac{\partial}{\partial \tau}\Big( P_{S} + \frac{2n}{\tau} \Big) &\leq& \Delta \Big( P_{S} + \frac{2n}{\tau} \Big) -2{\nabla}^{i}\Big( P_{S} + \frac{2n}{\tau} \Big){\nabla}_{i}v.
\end{eqnarray*}
This implies $\displaystyle \max_{M} \Big( P_{S} + \frac{2n}{\tau} \Big) = \displaystyle \max_{M} \Big( 2 \Delta v - |\nabla v|^{2} + {S} \Big)$ decreases as $\tau$ increases, which means that this quantity increases as $t$ increases.

\vfill

{\footnotesize
\noindent
{Hongxin Guo }\\
{School of mathematics and information science, Wenzhou University, \\
Wenzhou, Zhe-jiang 325035, China}\\
{\sc e-mail}: guo@wzu.edu.cn \\

{\footnotesize
\noindent
{Masashi Ishida}\\
{Department of Mathematics, Graduate School of Science,
Osaka University \\
1-1, Machikaneyama, Toyonaka, Osaka, 560-0043, Japan}\\
{\sc e-mail}: ishida@math.sci.osaka-u.ac.jp \\


\begin{thebibliography}{10}
%%%%%%%%%%%%%%%%%%%%%%%%%%%%





%\bibitem{And}
%Andrews, B., Harnack inequalities for evolving hypersurfaces, Math. Z. 217 (1994), 179--197.

%\bibitem{H-D-Cao}
%Cao, H.-D., On Harnack's inequalities for the Kahler-Ricci flow, Invent. Math. 109 (1992), 247--263.


%\bibitem{Cao-Ni}
%Cao, H.-D and Ni, L., Matrix Li--Yau--Hamilton estimates for the heat equation on K{\"{a}}hler manifolds, Math. Ann. 331 (2005), 795--807.

\bibitem{C}
Cao, X., Differential Harnack estimates for backward heat equations with potentials under the Ricci flow, J. Funct. Anal. 255 (2008), 1024--1038.


\bibitem{C-H}
Cao, X. and Hamilton, R., Differential Harnack estimates for time-dependent heat equations with potentials, Geom. Funct. Anal. 19 (2009), 989--100.




\bibitem{C-Z}
Cao, X. and Zhang, Z., Differential Harnack estimates for parabolic equations with potentials, Proceeding of complex and differential geometry (2011), 87--98


%\bibitem{Chow}
%Chow,B., On Harnack's inequality and entropy for the Gaussian curvature flow, Comm. Pure Appl. Math. 44 (1991), 469--483.


%\bibitem{Chow-Chu}
%Chow, B. and Chu,S.-C., A geometric interpretation of Hamilton's Harnack inequality for the Ricci flow, Math. Res. Lett. 2 (1995), 701--718.

\bibitem{S-F}
Fang, S., Differential Harnack inequalities for heat equations with potentials under the Bernhard List's flow, Geom Dedicata 161 (2012), 11--22.


\bibitem{G-F}
Guenther, C. M., The fundamental solution on manifolds with time-dependent metrics. J. Geom. Anal. 12 (2002), 425--436.

\bibitem{G-H}
{Guo, H. and He. T.}, Harnack estimates for the geometric flows, applications to Ricci flow coupled with harmonic map flow, Geom Dedicata (2013),
DOI 10.1007/s10711-013-9864-z.

\bibitem{G-I}
{Guo, H. and Ishida, M.}, Harnack Estimates for Nonlinear Heat Equations with potentials in Geometric Flows,  preprint (2013).

\bibitem{G-P-T}
{Guo, H., Philipowski, R. and Thalmaier, A.,} Entropy and lowest eigenvalue on evolving manifolds. Pacific J. Math., 264 (2013),  61--81.


\bibitem{ha-0}
{Hamilton, R.}, Three manifolds with positive Ricci curvature, {J. Differential Geom}. 17 (1982) 255--306.


%\bibitem{ha-10}
%{Hamilton, R.}, Four-manifolds with positive curvature operator, J. Differential Geom. 24 (1986), 153--179.

\bibitem{h-surface}
{Hamilton, R.}, The Ricci flow on surfaces, Contemp. Math. 71 (1988), 237-262, Amer. Math.
Soc., Providence, RI.



\bibitem{ha-1}
{Hamilton, R.}, The Harnack estimate for the Ricci flow. J. Differential Geom. 37 (1993), 225--243.

%\bibitem{ha-2}
%{Hamilton, R.},  A matrix Harnack estimate for the heat equation. Comm. Anal. Geom. 1 (1993), 113--126.

\bibitem{ha-2}
{Hamilton, R.},  Harnack estimate for the mean curvature flow, J. Differential Geom. 41 (1995), 215--226.

\bibitem{Ishida-1}
{Ishida, M.},  Geometric flows and differential Harnack estimates for heat equations with potentials, (preprint, 2012), submitted to a journal for publication.

\bibitem{K-Z}Kuang, S. and Zhang, Qi S., A gradient estimate for all positive solutions of the conjugate heat equation under Ricci flow.
J. Funct. Anal. 255 (2008), no. 4, 1008-1023.

\bibitem{Li-Yau}
Li, P. and Yau, S.-T, On the parabolic kernel of the Schrodinger operator. Acta Math. 156 (1986), 153--201.

\bibitem{List-1}
List, B., Evolution of an extended Ricci flow system, Comm. Anal. Geom 16 no. 5 (2008), 1007--1048.

\bibitem{Liu-1}
Liu, S., Gradient estimates for solutions of the heat equation under Ricci flow. Pacific J. Math. 243 (2009), 165--180.

\bibitem{Ma}
{Ma, L.}, Gradient estimates for a simple elliptic equation on non-compact Riemannian manifolds, J. Funct. Anal.241 (2006), 374-382.

\bibitem{Mu-1}
M{\"{u}}ller, R., Monotone volume formulas for geometric flows. J. Reine Angew. Math. 643 (2010), 39--57.

\bibitem{Mu-2}
M{\"{u}}ller, R., Ricci flow coupled with harmonic map flow. Ann. Sci. Ec. Norm. Super. (4) 45 (2012), 101--142

\bibitem{Ni-lei}
{Ni, L.},  Monotonicity and Li-Yau-Hamilton inequalities. In Surveys in differential geometry. Vol. XII. Geometric flows, Surv. Differ. Geom., 12, Int. Press, Somerville, MA, (2008) 251--301.

\bibitem{p-1}
Perelman, G., The entropy formula for the Ricci flow and its geometric applications, math.DG/0211159 (2002).

\bibitem{Sun}
Sun, J., Gradient estimates for positive solutions of the heat equation under geometric flow. Pacific J. Math. 253 (2011), 489--510.


\bibitem{Wu}
Wu, J-Y., Differential Harnack inequalities for nonlinear heat equations with potentials under the Ricci flow, Pacific J. Math. 257 (2012), 199--218

\bibitem{Z-1}
Zhang, Q. S., Some gradient estimates for the heat equation on domains and for an equation by Perelman. Int. Math. Res. Not. (2006), Art. ID 92314, 39 pp.



\end{thebibliography}
\end{document}